\documentclass[12pt]{amsart}

\usepackage{tikz,amsthm,amsmath,amstext,amssymb,amscd,epsfig,euscript, mathrsfs, dsfont,pspicture,multicol,graphpap,graphics,graphicx,times,enumerate,subfig,sidecap,wrapfig,color,pict2e}

\usepackage[colorlinks,citecolor=cyan,pagebackref,hypertexnames=false]{hyperref}

\usepackage{enumerate}

 \numberwithin{equation}{section}

\def\bB{{\mathbb{B}}}
\def\bC{{\mathbb{C}}}

\def\bR{{\mathbb{R}}}
\def\R{{\mathbb{R}}}

\def\bZ{{\mathbb{Z}}}

\def\bN{{\mathbb{N}}}

\def\cD{{\mathscr{D}}}
\def\cE{{\mathscr{E}}}

\def\cH{{\mathscr{H}}}

\def\ve{\varepsilon}

\renewcommand{\d}{{\partial}}
\def\lec{\lesssim}
\def\gec{\gtrsim}

\DeclareMathOperator{\diam}{diam}
\def\dim{\mathop\mathrm{dim}} 					
\def\dist{\mathop\mathrm{dist}} 						

\newcommand{\ps}[1]{\left( #1 \right)}
\newcommand{\bk}[1]{\left[ #1 \right]}
\newcommand{\ck}[1]{\left\{#1 \right\}}
\newcommand{\av}[1]{\left| #1 \right|}

\newcommand{\cnj}[1]{\overline{#1}}

\def\warrow{\rightharpoonup}

\def\XXint#1#2#3{{\setbox0=\hbox{$#1{#2#3}{\int}$ }
\vcenter{\hbox{$#2#3$ }}\kern-.58\wd0}}

\theoremstyle{plain}
\newtheorem*{maintheorem}{Main Theorem}
\newtheorem{theorem}{Theorem}

\newtheorem{corollary}[theorem]{Corollary}

\newtheorem{lemma}[theorem]{Lemma}

\theoremstyle{definition}

\newtheorem{definition}[theorem]{Definition}
\newtheorem{remark}[theorem]{Remark}

\numberwithin{equation}{section}
\numberwithin{theorem}{section}

  \DeclareFontFamily{U}{mathb}{\hyphenchar\font45} 
\DeclareFontShape{U}{mathb}{m}{n}{
      <5> <6> <7> <8> <9> <10> gen * mathb
      <10.95> mathb10 <12> <14.4> <17.28> <20.74> <24.88> mathb12
      }{}
\DeclareSymbolFont{mathb}{U}{mathb}{m}{n}

\DeclareMathSymbol{\toitself}{3}{mathb}{"FD}  

\begin{document}

\title[Dimension drop for harmonic measure]{Dimension drop for harmonic measure on Ahlfors regular boundaries}

\author[Azzam]{Jonas Azzam}

\address{Jonas Azzam\\
School of Mathematics \\ University of Edinburgh \\ JCMB, Kings Buildings \\
Mayfield Road, Edinburgh,
EH9 3JZ, Scotland.}
\email{j.azzam "at" ed.ac.uk}

\begin{abstract}
We show that given a domain $\Omega\subseteq \R^{d+1}$ with uniformly non-flat Ahlfors $s$-regular boundary with $s\geq d$, the dimension of its harmonic measure is strictly less than $s$. 
\end{abstract}


\maketitle

\tableofcontents

\section{Introduction}

The purpose of this note is to study the dimension of harmonic measure $\omega_{\Omega}$ for a  connected domain  $\Omega\subseteq \R^{d+1}$. Naturally, $\dim \omega_{\Omega} \leq \dim \d\Omega$, but the inequality can be strict. For example, it is a classical result of Jones and Wolff \cite{JW88} that if $\Omega\subseteq \bC$, then the dimension is always at most 1, even if $\dim \d\Omega>1$ (which improves on an earlier result of Makarov for simply connected planar domains \cite{Mak85}). In higher dimensions, the analogous property is no longer true: there are domains called {\it Wolff snowflakes} in $\R^{d+1}$ whose harmonic measure can be strictly larger or strictly less than $d$; the $d=2$ case is due to Wolff \cite{Wol95}, and the general case is Lewis, Verchota and Vogel in \cite{LVV05} (note that even though the dimension can be above $d$, a result of Bourgain says that the dimension harmonic measure for any domain in $\bR^{d+1}$ can't get too close to $d+1$ \cite{Bou87}, and it is an open problem to determine what the supremal dimension can be). While these are all very non-trivial results, these {\it Wolff snowflakes} actually have some nice geometry. In particular, they are two-sided uniform domains. We say a domain $\Omega$ is {\it $C$-uniform} if for all $x,y\in \Omega$ there is a curve $\gamma\subseteq \Omega$ so that 
\[
\cH^{1}(\gamma)\leq C|x-y|\]
and
\[
\dist(z,\Omega^{c})\geq C^{-1} \min\{\ell(x,z),\ell(y,z)\}
\]
where $\ell(a,b)$ denotes the length of the subarc of $\gamma$ between $a$ and $b$. A domain is {\it two-sided} uniform if both $\Omega$ and $\cnj{\Omega}^{c}$ are $C$-uniform domains for some $C$. 

Since two-sided uniform domains have boundaries with dimension at least $d$ and a Wolff snowflake can have dimension less than $d$, this example also shows that a dimension drop for harmonic measure can occur for a domain that is quite nice in terms of its connectivity and boundary properties. Hence, it is an interesting problem to identify some general criteria for when a whole class of domains $\Omega$ satisfy $\dim \omega_{\Omega}<\dim \d\Omega$. 

A dimension drop for harmonic measure occurs for some domains whose boundaries have some self-similar structure. This phenomenon was first observed by Carleson \cite{Car85} for complements of planar Cantor sets whose boundaries have dimension at least 1 (rather, he showed for a particular class of Cantor sets $C$, $\dim\omega_{\Omega}<1$). Later, Jones and Wolff showed the same result but for uniformly perfect sets satisfying a certain uniform disconnectedness property (see \cite{JW88} or \cite[Section X.I.2]{Harmonic-Measure}). Makarov and Volberg showed $\dim \omega_{\Omega}<\dim \d\Omega$ when $\d\Omega$ belongs to a more general class of Cantor sets (with $\dim \d\Omega$ possibly below $1$) \cite{MV86} and then to Cantor repellers of any dimension \cite{Vol92}, that is, sets $K$ for which there are smooth disjoint domains $U_{i}\subseteq \bC$ compactly contained in a domain $U\subseteq \bC$ and univalent maps $f_{i}:U_{i}\rightarrow U$ for which $K$ is the unique compact set such that $K=\bigcup f_{i}(U)$ (and in Volberg's result, two of the maps need to be linear). Urba'{n}ski and Zdunik have also shown that the attractors of conformal iterated function systems (IFS) have a dimension drop when either the limit set is contained in a real-analytic curve, if the IFS consists of similarities only, or if the IFS is irregular (see \cite{UZ02}). See also \cite{Mey09,Pop98}.

A common thread to many of these results is etiher some uniform disconnectivity property (see equations (XI.2.1)-(XI.2.3) in \cite{Harmonic-Measure}, \cite{JW88}, \cite[Lemma 2.5]{Bat96}, and \cite[Lemma 5]{Car85}), or some self-similar or ``dynamically defined" structure \cite{MV86,Vol92,Vol93}. In the latter case, self-similarity allowed authors to exploit some ergodic theory, except the work of Batakis \cite{Bat96} which gave a non-ergodic proof of dimension drop for a wide class of Cantor sets  that also works in higher dimensions, and later studied how the dimension is continuous with respect to the parameters defining the Cantor set \cite{Bat00,Bat06}. 

In the present paper, we develop a different sufficient condition for when the harmonic measure is strictly less than the dimension of the boundary. Though it assumes some strong conditions on the Hausdorff measure on the boundary, it requires no self-similar structure or uniform-disconnectedness.  Instead, it assumes some uniform non-flatness condition (which will hold for any self-similar set in $\R^{d+1}$ of dimension at least $d$ that isn't a $d$-dimensional plane), and also holds in higher dimensions.

\begin{maintheorem}
Given $d\in \bN$, $C_{1}>0$, and $\beta>0$, there are constants $s_{0}<d$ and $\kappa\in (0,1)$ so that the following holds. Let $s_{0}<s\leq d+1$ and $\Omega\subseteq \R^{d+1}$ be a connected domain whose boundary is $C_{1}$-Ahlfors $s$-regular, meaning if $\sigma=\cH^{s}|_{\d\Omega}$, then 
\begin{equation}
\label{e:adr}
C_{1}^{-1} r^{s} \leq \sigma(B(x,r)) \leq C_{1}r^{s} \mbox{ for all }x\in \d\Omega, \;\; 0<r<\diam \d\Omega. 
\end{equation}
Also suppose there is $\beta>0$ so that 
\def\bbeta{b\beta}
\begin{multline}
\label{e:bigbeta}
\bbeta_{\omega_{\Omega}}(x,r):=\inf_{V}\bk{\sup_{y\in B(x,r)\cap \d\Omega}\frac{\dist(x,V)}{r}+\sup_{y\in V\cap B(x,r)}\frac{\dist(y,\d\Omega)}{r}}\\
\geq \beta>0
 \end{multline}
 where the infimum is over all $d$-dimensional planes $V\subseteq \R^{d+1}$. Then $\dim \omega_{\Omega}<\kappa s$, meaning there is a set $K$ with $\dim K<\kappa s$ so that $\omega_{\Omega}(K^{c})=0$. 

\label{thmi}
\end{maintheorem}

\def\maintheorem{{\hyperref[thmi]{Main Theorem }}}

Some remarks are in order. Firstly, this theorem does not cover all the fractals considered by Batakis, which may or may not satisfy \eqref{e:adr}. Secondly, nor does it tackle sets all sets with dimension $s<d$. However, the domains Batakis and others considered need their boundaries to be totally disconnected or be defined in some recursive way, whereas the domains we consider can be connected and quite random. Also notice that if $s>d$, then \eqref{e:adr} implies there is $\beta>0$ depending on $s$ so that \eqref{e:bigbeta} holds, so \eqref{e:bigbeta} is not needed in this case to guarantee a dimension drop; however, if \eqref{e:bigbeta} holds for some $\beta$ and $s>d$ is close enough to $d$ depending on $\beta$, we can have that $\dim \omega_{\Omega}<d$.

Some domains with $2$-dimensional boundaries  in $\R^{3}$ not covered by previous results that are admissible for the above theorem are the complements of the tetrahedral Sierpinski Gasket and $C\times[0,1]$ where $C$ is the $1$-dimensional $4$-corner Cantor set. Note their boundaries are either connected or have large connected components. Some boundaries with no self-similar structure include bi-Lipschitz images of these sets, or a snowflaked image of $\R^{2}$ in $\R^{3}$ (that is, the image of a map $f:\R^{2}\rightarrow \R^{3}$ satisfying $|f(x)-f(y)|\sim |x-y|^{1/s}$ with $s>1$).

The proof of the \maintheorem is modelled after that of \cite{Bat96} and relies on a trick introduced by Bourgain in \cite{Bou87}. Similar to those papers, the first step we need to take is to show that inside any ``cube" on the boundary of our domain, the $s$-dimensional density of harmonic measure dips (or increases) inside a sub-cube of comparable size (compare the bottom of page 480 in  \cite{Bou87} or  \cite[Lemma 2.7]{Bat96}). After showing this, the proof is similar to those above: this dip causes harmonic measure to concentrate elsewhere in the cube (see \cite[Lemma 2]{Bou87} or \cite[Lemma 2.8]{Bat96}) and one can iterate this to show that harmonic measure is in fact supported on a set of dimension less than $s$. 

This density drop is easier to show when $s>d$ using a touching-point argument, and is more quantitative. To prove the density drop allowing for $s=d$ (or slightly smaller than $d$) assuming non-flatness, we use a compactness argument to show that, if this weren't true, then we could find a domain with $d$-regular boundary such that the density of its harmonic measure was uniformly bounded over all small balls on the boundary, but then harmonic measure would be absolutely continuous with respect to $d$-dimensional Hausdorff measure. This gives us a lot of structural information by the following result:

\begin{theorem}\label{t:AHM3TV}
\cite[Theorem 1.1]{AHMMMTV16} Let $d\geq 1$ and $\Omega\subsetneq\R^{d+1}$ be an open  connected set
and let $\omega:=\omega^p$ be the harmonic measure in $\Omega$ where $p$ is a fixed point in $\Omega$.
Let $E\subset\partial\Omega$ be a subset with Hausdorff measure $\cH^d(E)<\infty$. Then:
\begin{itemize}
\item[(a)] If $\omega$ is absolutely continuous with respect to $\cH^d$ on $E$, then $\omega|_E$ is $d$-rectifiable, in the sense that  $\omega$-almost all of $E$ can be covered by a countable union of $n$-dimensional Lipschitz graphs.

\item[(b)]  If $\cH^d$ is absolutely continuous with respect to $\omega$ on $E$, then $E$ is a $d$-rectifiable set, in the sense that  $\cH^d$-almost all of $E$ can be covered by a countable union of $d$-dimensional Lipschitz graphs.
\end{itemize}
\end{theorem}

Thus, the boundary of our domain has tangents, but this violates \eqref{e:bigbeta}.\\

Using compactness arguments for harmonic measure is quite common, and the author first learned of it from the work of Kenig and Toro \cite{KT99}. For more recent applications, see \cite[Section 3]{HMMTZ17}, \cite{BE17}, and \cite{AM15}, and the references therein, which are primarily concerned with uniform domains. See also \cite{AMT17} for an example of compactness arguments used in non-uniform domains.

We don't know whether our result holds for all $s\in (d-1,d+1]$. If $s<d$, the answer to this question relies on knowing whether harmonic measure is always singular with respect to $\cH^{s}$-measure if the boundary is Ahlfors $s$-regular. As far as the author knows ,this is an open question. An answer in the affirmative would mean that the arguments here could be used again to obtain the whole range $(d-1,d+1]$. \\

We would like to thank Mihalis Mourgoglou and Xavier Tolsa for their comments on the manuscript, and for the anonymous referee for spotting several errors and to whom we are indebted for making the paper much clearer.

\section{Preliminaries}

We will let $B(x,r)=\{y: |x-y|\leq r\}$ and $\bB=B(0,1)$. If $B=B(x,r)$, we let $\lambda B=B(x,\lambda r)$, $x_{B}=x$, and $r_{B}=r$. We will denote by $\cH^{s}$ and $\cH^{s}_{\infty}$ the $s$-dimensional Hausdorff measure and Hausdorff content respectively. For a reference on geometric measure theory and Hausdorff measure, see \cite{Mattila}. 

We will write $a\lesssim b $ if there is a constant $C>0$ so that $a\leq C b$ and $a\lesssim_{t} b$ if the constant depends on the parameter $t$. As usual we write $a\sim b$ and $a\sim_{t} b$ to mean $a\lesssim b \lesssim a$ and 
$a\lesssim_{t} b \lesssim_{t} a$ respectively. We will assume all implied constants depend on $d$ and hence write $\sim$ instead of $\sim_{d}$.

Whenever $A,B\subset\mathbb{R}^{d+1}$ we define
\[
\mbox{dist}(A,B)=\inf\{|x-y|;\, x\in A, \, y\in B\}, \, \mbox{and}\, \, \mbox{dist}(x,A)=\mbox{dist}(\{x\}, A). 
\]
Let $\diam A$ denote the diameter of $A$ defined as
\[
\diam A=\sup\{|x-y|;\, x,y\in A\}.
\]

For a domain $\Omega$ and $x\in \Omega$, we let $\omega_{\Omega}^{x}$ denote the harmonic measure for $\Omega$ with pole at $x$ and $G_{\Omega}(\cdot,\cdot)$ the associated Green function. For a reference on harmonic measure and the Green function, see \cite{AG}.

Given a domain $\Omega$ and a ball $B$ centered on $\d\Omega$, we say $x$ is a {\it $c$-corkscrew point} for $B\cap \Omega$ if $B(x_{B},cr_{B})\subseteq B\cap \Omega$. 

We will say a domain $\Omega$ has {\it lower $s$-content regular complement} with constant $c_{1}$ if for all $B$ centered on $\d\Omega$ and $0<r_{B}<\diam\d\Omega$,
\begin{equation}
\label{e:lowerd}
\cH^{d}_{\infty}(B\backslash \Omega)\geq c_{1} r_{B}^{d}
\end{equation}

The following lemma is due to Bourgain for $\bR^{3}$ in \cite[Lemma 1]{Bou87}. The proof in $\bR^{d+1}$ is identical and shown in \cite[Lemma 3.4]{AHMMMTV16}.

\begin{lemma} 
\label{l:bourgain}
If $\Omega\subseteq \R^{d+1}$ has lower $s$-content regular complement with constant $c_{1}$ and $s>d-1$, then there is $b\in (0,1)$ so that 
\begin{equation}
\label{e:bou}
\omega_{\Omega}^{x}(B(\xi,r))\gec c_{1} \;\;\mbox{ for }\xi\in \d\Omega, \;\; 0<r<\diam \d\Omega,\;\; \mbox{ and }x\in B(\xi,b r).
\end{equation}
\end{lemma}

A domain $\Omega$ satisfies the {\it capacity density condition (CDC)} if, for all $x\in \d\Omega$ and $0<r<\diam \d\Omega$, 
\[\textup{Cap}({B}(x,r) \cap \Omega^c, B(x,2r)) \gtrsim r^{n-1},\] 
where Cap$(\cdot,\cdot)$ stands for the variational $2$-capacity of the condenser $(\cdot,\cdot)$ (see \cite[p. 27]{HKM} for the definition).

\begin{remark}
We will state some lemmas below that assume the CDC, but keep in mind the CDC is implied by \eqref{e:lowerd} for $s>d-1$. This can be seen from \cite[Lemma 2.31]{HKM}). Alternatively, Ancona showed in \cite[Lemma 3]{Anc86} that the CDC is equivalent to the property that, for some $\alpha>0$,
\[
\omega_{\Omega\cap B}^{x}(\d B\cap \Omega) \lec \ps{\frac{|x-x_{B}|}{r}}^{\alpha} \;\;\mbox{ for }x\in B\cap \Omega
\]
for all balls $B$ centered on $\d\Omega$ with $r_{B}<\diam \d\Omega$. This property is implied by Bourgain's lemma (see for example \cite[Lemma 2.3]{AM18}).  In particular, the maximum principle implies the following lemma.
\end{remark}

\begin{lemma} \label{l:holder}
Let $\Omega\subsetneq\R^{d+1}$ be an open set that satisfies the CDC and let $x\in \d\Omega$. Then there is $\alpha>0$ so that for all $0<r<\diam(\Omega)$,
\begin{equation}\label{e:holder}
 \omega_{\Omega}^{y}({B}(x,r)^{c})\lesssim \ps{\frac{|x-y|}{r}}^{\alpha},\quad \mbox{ for all } y\in \Omega\cap B(x,r),
 \end{equation}
where $\alpha$ and the implicit constant depend on $n$ and the CDC constant.
\end{lemma}

\begin{lemma}
\cite[Lemma 1]{Aik08}
For $x\in \Omega\subseteq \bR^{d+1}$ and $\phi\in C_{c}^{\infty}(\bR^{d+1})$, 
\begin{equation}
\label{e:ibp}
\int \phi\omega_{\Omega}^{x} =\phi(x)+ \int_{\Omega} \triangle \phi(y) G_{\Omega}(x,y)dy.
\end{equation}
\end{lemma}
We will typically use the above lemma when $\phi(x)=0$. 
\begin{lemma}
\label{e:green}
Let $\Omega\subseteq \bR^{d+1}$ be a CDC domain. Let $B$ be a ball centered on $\d\Omega$ with $0<r_{B}<\diam\d\Omega$. Then
\begin{equation}
\label{e:G<w}
G_{\Omega}(x,y)\lec  \frac{\omega_{\Omega}^{x}(4B)}{r_{B}^{d-1}} \mbox{ for all }x\in \Omega\backslash 2B \mbox{ and }y\in B
\end{equation}
\end{lemma}

The above lemma has evolved over many years and this isn't the most general statement, but it will suit our purposes. It follows from the proof of Lemma 3.5 in \cite{AH08}: it is assumed there that the domain is John (and so in particular bounded) but is not necessary for the above statement. A version of this is also shown in \cite{AHMMMTV16} that works for general bounded domains without the CDC (and implies the above inequality), but it is only for $d>1$.

Recall that a Harnack chain between two points $x,y\in \Omega$ is a sequence of balls $B_{1},...,B_{n}$ for which $2B_{i}\subseteq \Omega$. By Harnack's inequality,  there is $M>0$ so that for any non-negative harmonic function $u$ on $\Omega$, 
\begin{equation}
\label{e:harnack}
u(x)\leq M^{n} u(y).
\end{equation}

\begin{lemma}\label{l:limlem}
\cite[Lemma 2.9]{AMT17} Let $\Omega_{j}\subset \bR^{d+1}$ be a sequence of domains with lower $s$-content regular complements, $s>d-1$, $0\in \d\Omega_{j}$, $\inf \diam \d\Omega_{j}>0$, and suppose there is a ball $B_0=B(x_{0}, r_0) \subset \Omega_j$ for all $j \geq1$. Then there is a connected open set $\Omega_{\infty}^{x_{0}}$ containing $B$ so that, after passing to a subsequence,
\begin{enumerate}
\item $G_{\Omega_{j}}(x_{0},\cdot)\rightarrow G_{\Omega_{\infty}^{x_{0}}}(x_{0},\cdot)$ uniformly on compact subsets of $\{x_{0}\}^{c}$,
\item   $\omega_{\Omega_{j}}^{x_{0}}\warrow \omega_{\Omega_{\infty}^{x_{0}}}^{x_{0}}$, and
\item  $\Omega_{\infty}^{x_{0}}$ also has lower $s$-content regular complement with the same constant.
\end{enumerate}
\end{lemma}

This is not how it is stated in \cite{AMT17}, but it follows from the proof. Indeed, the lemma is stated for CDC domains, but they use the fact that CDC domains have lower $s$-content regular complements for a particular dimension and constant, and then use that characterization to prove the lemma.

\section{Rectifiability}

\begin{definition}
A set $E\subseteq \R^{n}$ is {\it $d$-rectifiable} if $\cH^{d}(E)<\infty$ and it may be covered up to $\cH^{d}$-measure zero by a countable union of rotated Lipschitz graphs. 
\end{definition}

For two measures $\mu$ and $\nu$ and a ball $B$, let 
\[
F_{B}(\mu,\nu)=\sup_{f} \av{\int fd\mu-\int fd\nu}
\]
where the supremum is over all nonnegative $1$-Lipschitz functions supported in $B$. 

For $x\in \R^{n}$, $r>0$, $\mu$ a Radon measure, and an affine $d$-dimensional plane $V$, let 
\[
\alpha_{\mu}(x,r,V)=r^{-d-1}\inf_{c>0}F_{B(x,r)}(\mu,c\cH^{d}|_{V})
\]
and let 
\[
\alpha_{\mu}(x,r)=\inf_{V} \alpha_{\mu}(x,r,V)\]
where the infimum is over all $d$-dimensional affine planes $V$. 

\begin{lemma}
\label{l:xavi}
\cite[Lemma 2.1]{Tol15}
If $\mu$ is a Radon measure and $\Gamma$ is a Lipschitz graph, then 
\begin{equation}
\label{e:xavi}
\int_{0}^{1} \alpha_{\mu}(x,r)^{2} \frac{dr}{r} <\infty \mbox{ for $\cH^{d}$-a.e. $x\in \Gamma$}.
\end{equation}
\end{lemma}

As a corollary, if we let $\alpha_{E}:=\alpha_{\cH^{d}|_{E}}$, we get the following. 

\begin{corollary}
\label{c:xavi}
If $E$ is a Borel set and $F\subseteq E$ is $d$-rectifiable, then 
\[
\lim_{r\rightarrow 0}\alpha_{E}(x,r)=0 \;\; \mbox{ for $\cH^{d}$-a.e. $x\in F$}.
\]
\end{corollary}

This follows from Lemma \ref{l:xavi} since $F$ may be covered by countably many $d$-dimensional Lipschitz graphs and since $\alpha_{\mu}(x,r)\leq (s/r)^{d+1}\alpha_{\mu}(x,s)$ for $r<s$. 

\begin{remark}
The proof of the corollary is much simpler than envoking Lemma \ref{l:xavi}, and is quite standard, but we couldn't find a short reference for it. It can actually be proven more simply by the techniques in Chapters 14-16 of \cite{Mattila}, although for the sake of brevity we didn't want to recall too much background in order to do this.
\end{remark}

\section{Non-flatness or big dimension implies change in density}

\def\bbeta{b\beta}

For a Radon measure $\mu$, $s\geq 0$, and a ball $B$, we define
\[
\Theta_{\mu}^{s}(B)= \frac{\mu(B)}{r^{s}}.
\]

Throughout this section, we will work with a domain $\Omega\subseteq \R^{d+1}$ and we will let $\omega=\omega_{\Omega}$ for short. 

\begin{lemma}
\label{l:bigD-s>d}
Suppose $s>d\in \bN$, $c,C_1,M>0$, and $\Omega\subseteq \R^{d+1}$ has lower $s$-content regular with constant $C_1$. Let $B$ be centered on $\d\Omega$, $A>4$, $A_0>0$, and $x\in B\backslash 2A^{-1} B$ a $c$-corkscrew point in $B$ (meaning $x\in B\cap \Omega$ and $B(x,2cr_{B})\subseteq \Omega$) so that 
\[
\omega^{x}(A^{-1}B)\geq  A_{0}.
\]
Then there is $\delta=\delta(d,c,M,C_1,A_0,A)>0$ and a ball $B'\subseteq 5A^{-1}B$ with $r_{B'}\geq \delta r_{B}$ so that 
\[
\Theta_{\omega^{x}}^{s}(B')>Mr_{B}^{-s}
\]
\end{lemma}

\begin{proof}
Without loss of generality, $B=\bB$. We first use a method of Aikawa and Hirata \cite{AH08} to show there is a Harnack chain from $x$ to a point in $2A^{-1}\bB$. Let $\phi$ be a smooth function supported on $2A^{-1}\bB$ and equal to $1$ on $A^{-1}\bB$. Then 
\[
A_0\leq \int \phi d\omega^{x}
\stackrel{\eqref{e:ibp}}{=}\int G(x,y)\triangle \phi(y)dy
\lec \sup_{y\in 2A^{-1}\bB}G(x,y) .
\]
Thus, there is a universal constant $\lambda>0$ so that
\[
y\in E_{\lambda}=\{z: G(x,z)>\lambda A_0\}.
\]
Note that as $G(x,z)\lec |x-z|^{1-d}$, $E_{\lambda}$ is a bounded set with diameter depending on $\lambda$. Moreover, $E_{\lambda}$ is open and also contains $x$. By the maximum principle, it must also be connected, so there is a curve $\gamma\subseteq E_{\lambda}$ joining $x$ to a point in $2A^{-1}B$. Note that if $y\in \gamma$ $\delta_{\Omega}(y)<\ve$, and $\xi\in \d\Omega$ is closest to $y$, then $x\not\in B(\xi,c/2)$, and so 
\[
G(x,y)\stackrel{\eqref{e:holder}}{\lec} ||G(x,\cdot)||_{L^{\infty}(B(\xi,c/4))}\ve^{\alpha}
\stackrel{\eqref{e:G<w}}{\lec} \omega^{x}(B(\xi,c))c^{1-d}\ve^{\alpha}\]
and so $y\not\in E_{\lambda}$ for $\ve$ small enough (depending on $A$ and $d$). Thus, $\dist(y,\d\Omega)\geq \ve>0$ for all $y\in \gamma$. Thus, we can find a Harnack chain from $x$ to a point $y\in \gamma\cap 2A^{-1}\bB$ of uniformly bounded length (depending on $\ve,d,$ and $A$).

Again, let $\xi$ be the closest point in $\d\Omega$ to $y$. Let $v=\frac{y-\xi}{|y-\xi|}$ and for $t>0$ let $B_{t}=B(\xi,t)$ (see Figure \ref{f:figure}). 

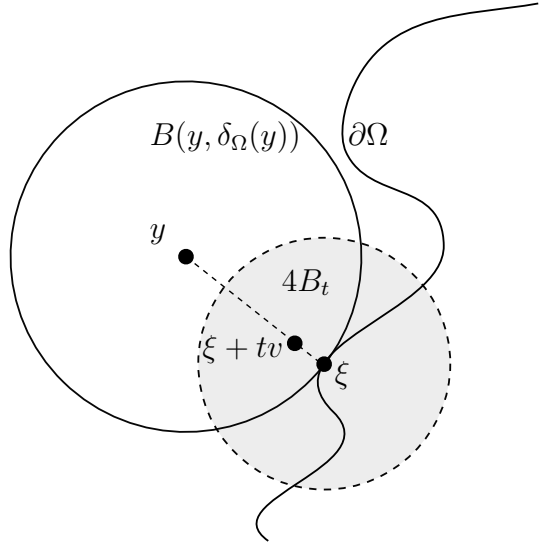
\begin{figure}
\begin{center}
\includegraphics[width=200pt]{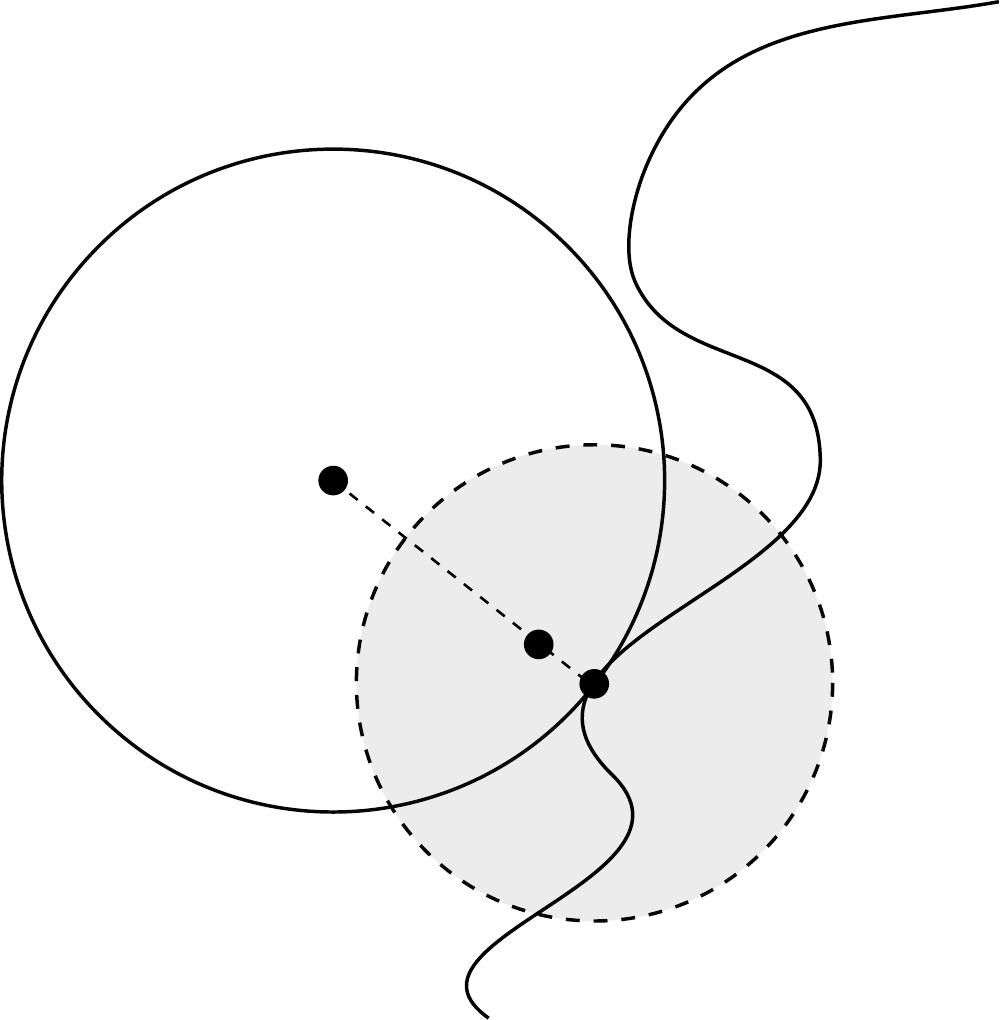}
\begin{picture}(0,0)(200,0)
\put(125,150){$\d\Omega$}
\put(50,150){$B(y,\delta_{\Omega}(y))$}
\put(50,115){$y$}
\put(120,60){$\xi$}
\put(70,70){$\xi+tv$}
\put(100,95){$4B_{t}$}
\end{picture}
\end{center}
\caption{The balls $4B_{t}$ and $B(y,\delta_{\Omega}(y))$. Observe that $B(y,\delta_{\Omega}(y)$ is tangent to $\d\Omega$. }
\label{f:figure}
\end{figure}

Recall that if $d>1$, for any $r>0$ and $z\in B(y,r)\backslash \{y\}$,
\[
G_{B(y,r)}(y,z) = c(|y-z|^{1-d}-r^{1-d})
\]
where $c>0$ is some constant depending on $d$. (If $d=1$, Green's function is instead a multiple of $\log\frac{r}{|y-z|}$, and the estimates below are similar in this case). Thus, for $t>0$ small, and if $d>1$,
\begin{align*}
\omega^{x}(4B_{t})
& \stackrel{\eqref{e:G<w}}{\gec} G(x,\xi+tv)t^{d-1} 
\stackrel{\eqref{e:harnack}}{\sim}_{\ve} G(y,\xi+tv)t^{d-1}\\
& \geq G_{B(y,\delta_{\Omega}(y))}(y,\xi+tv)t^{d-1}\\
& = c \ps{|y-(\xi+tv)|^{1-d} - \delta_{\Omega}(y)^{1-d}}t^{d-1}\\
& = c \ps{(\delta_{\Omega}(y)-t)^{1-d} - \delta_{\Omega}(y)^{1-d}}t^{d-1}\\
& = c \ps{(1-t/\delta_{\Omega}(y))^{1-d} -1}\frac{t^{d-1}}{\delta_{\Omega}(y)^{d-1}}
 \gec \frac{t^{d}}{\delta_{\Omega}(y)^{d}}.
\end{align*}

Note that as there is a Harnack chain of length depending on $\ve$ between $x$ and $y$, and because $x$ is a $c$-corkscrew point in $B$, we have that 
\[
\delta_{\Omega}(y)\sim_{\ve}\delta_{\Omega}(x)\sim_{c} r_{B}=1.
\]

In particular, for $s>d$, the above estimates imply $\omega^{x}(4B_{t})\gec t^{d}$. In particular, $\Theta_{\omega^{x}}^{s}(4B_{t })
\gec t^{d-s}$, hence for $t$ small (depending on $M$), $\Theta_{\omega^{x}}^{s}(4B_{t})> M$. Since $y\in 2A^{-1}B$ and the center of $B$ is in $\d\Omega$, $|\xi-y|\leq 2A^{-1}$, so for $t$ small enough, we can also guarantee that $4B_{t }\subseteq 5A^{-1}\bB$, and so $4B_{t}$ is our desired ball.

\end{proof}
\def\bbeta{b\beta}

\begin{lemma}
\label{l:bigD-s=d}
Given $d\in \bN$, $M,c,C_{1}>0$, and $\beta>0$, there is $s_{0}<d$ so that the following holds. Let $s_{0}<s\leq d+1$ and $\Omega\subseteq \R^{d+1}$ be a connected domain so that \eqref{e:adr} and \eqref{e:bigbeta} hold. Let $A>4$, $A_{0}>0$ and suppose $x\in B\backslash 2A^{-1}B$ is a $c$-corkscrew point in $B$ so that 
\[
\omega^{x}(A^{-1}B^{\circ})\geq  A_{0}.
\]
Then there is $\delta=\delta(d,c,C_1,M,\beta,A_0,A)>0$ and a ball $B'\subseteq 5A^{-1} B$ so that $r_{B'}\geq \delta r_{B}$ and
\[
\Theta_{\omega^{x}}^{s}(B')>Mr_{B}^{-s}.
\]
\end{lemma}

\begin{proof}
Again, we can assume $B=\bB$. Suppose instead that for all $j\in \bN$ we could find a domain $\Omega_{j}$ with $C_{1}$-Ahlfors $s_j$-regular boundary with $d-\frac{1}{j}<s_j\leq d+1$ containing $0$ and $x_{j}$ a $c$-corkscrew point in $\bB\backslash 2A^{-1}\bB$ so that for all $j$,
\[
\omega_{\Omega_{j}}^{x_{j}}(A^{-1}\bB^\circ)\geq A_{0}
\]
yet for all $B\subseteq \frac{1}{2}\bB$ with $r_{B}\geq \frac{1}{j}$,
\[
\omega_{\Omega_{j}}^{x_{j}}(B)\leq Mr_{B}^{s_j}. 
\]
Since the $\d\Omega_{j}$ are Ahlfors $s_j$-regular with constant $C_{1}$ and $s_j\rightarrow s$, for $j$ large they all uniformly have large $(d-1/2)$-content regular complements. By \ref{l:limlem}, we may pass to a subsequence so that $s_{j}\rightarrow s\in [d,d+1]$, $x_{j}\rightarrow x_0\in \bB\backslash 2A^{-1}\bB$ a $c$-corkscrew point in $\bB$ for a domain $\Omega_0=\Omega_{\infty}^{x_{0}}$ so that $\omega_{\Omega_{j}}^{x_{j}}\warrow \omega_{\Omega_0}^{x_{0}}$ and also has lower $(d-1/2)$-content regular complement. In particular, one can show that 
\begin{equation}
\label{e:w0dfrostmann}
\omega_{\Omega_{0}}^{x_{0}}(B)\leq Mr_{B}^{d}
\end{equation}
for all balls $B$ centered on $\d\Omega_{0}$ contained in $5A^{-1}\bB$, and 
\begin{equation}
\label{e:omega0dbig}
\omega_{\Omega_{0}}^{x_{0}}(A^{-1}\bB^{\circ})\geq A_{0}.
\end{equation}

Hence, by the previous lemma, \eqref{e:w0dfrostmann} is impossible if $s>d$, thus we must have $s=d$. In particular, $\omega_{\Omega_{0}}^{x_{0}}\ll \cH^{d}$ in $5A^{-1}\bB$. 

Since the $\d\Omega_{j}$ are Ahlfors $s_j$-regular with constant $C_{1}$ and $s_j\rightarrow s$, we can also pass to a subsequence so that $\d\Omega_{j}\cap 10\bB$ converges in the Hausdorff metric to a set $\Sigma$ so that if $\sigma=\cH^{d}|_{\Sigma}$, then $\sigma(B)\sim r^{s}$ for all $B\subseteq 10\bB$ centered on $\Sigma$. 

We will now show that for this new set we have
\begin{equation}
\label{e:sigma-no-flat}
\bbeta_{\Sigma}(x,r)\geq \beta>0 \mbox{ for all }x\in \Sigma\cap \bB, \;\; 0<r<1.
\end{equation}

We will prove by contrapositive. Suppose instead that there was $x\in \Sigma\cap \bB$ and $0<r<1$ so that $\bbeta_{\Sigma}(x,r)< \beta$, so there is a $d$-plane $P$ so that 
\[
\sup_{y\in B(x,r)\cap P} \dist(y,\Sigma) + \sup_{y\in B(x,r)\cap \Sigma}\dist(y,P)<\beta.
\]
If $z_j\in B(x,r)\cap \d\Omega_j$ is farthest from $P$ and $y_j\in P\cap B(x,r)$ is farthest from $\d\Omega_j$, then we can pass to a subsequence so that $z_j\rightarrow z\in B(x,r)\cap \Sigma$ and $y_j\rightarrow  y \in P\cap B(x,r)$, and then
\begin{align*}
\beta & \leq \limsup_{j\rightarrow\infty}
 \ps{\sup_{p\in \d\Omega_j\cap B(x,r)}\dist(p,P)+\sup_{q\in P\cap B(x,r)}\dist(q,\d\Omega_j)}\\
& =\limsup_{j\rightarrow\infty} \ps{ \dist(z_j,P) + \dist(y_j,\d\Omega_j)}\\
& =\dist(z,P)+\dist(y,\Sigma)<\beta,
\end{align*}
which is a contradiction, and this proves \eqref{e:sigma-no-flat}.

Thus, $\d\Omega_{0}\cap 5A^{-1}\bB\subseteq \Sigma$, so $\cH^{d}(\d\Omega_{0}\cap 5A^{-1}\bB)<\infty$. By \eqref{e:w0dfrostmann} and \eqref{e:omega0dbig}, and because  $\omega_{\Omega_{0}}^{x_{0}}\ll \cH^{d}$ on $\d\Omega_{0}\cap 5A^{-1} \bB$, there is $E\subseteq \d\Omega_{0}\cap 5A^{-1}\bB$ with $0<\cH^{d}(E)<\infty$ so that $\cH^{d}\ll \omega_{\Omega_{0}}^{x_{0}}\ll \cH^{d}$ on $E$. To see this, just observe that if $f=d\omega_{\Omega}^{x_{0}}|_{5A^{-1}\bB\cap \d\Omega_{0}}/d\cH^{d}|_{5A^{-1}\bB\cap \d\Omega_{0}}$ is the Radon-Nikodym derivative, then $E=\{x\in 5A^{-1}\bB\cap \d\Omega_{0}: f(x)>0\}$ is our desired set (see \cite[Theorem 2.12]{Mattila}). Now Theorem \ref{t:AHM3TV} implies that $E$ is a $d$-rectifiable set of positive $\cH^{d}$-measure. Since $E\subseteq \Sigma$, Corollary \ref{c:xavi} implies that 
\[
\lim_{r\rightarrow 0} \alpha_{\Sigma}(x,r)=0.
\]
for some $x\in E$. In particular, if $\ve>0$, there is $r>0$ small enough and a plane $V$ so that 
\begin{equation}
\label{e:alpha4rV}
\alpha_{\Sigma}(x,4r,V)<\epsilon.
\end{equation}
Without loss of generality, we can assume $V\cap B(x,4r)\neq\emptyset$.  Indeed, suppose instead that $V\cap B(x,4r)\emptyset$. Let $\phi(z)=(4r-|x-z|)_{+}$, then $\phi(z)=0$ on $V$ and $\phi\geq 2r$ on $B(x,2r)$. This and the fact that  $\Sigma$ is Ahlfors regular imply
\begin{multline}
\ve>\alpha_{\Sigma}(x,4r,V)\geq (4r)^{-d-1} \int \phi(z)d\sigma(z)
\\ \geq (4r)^{-d-1} \sigma(B(x,2r)) 2r 
\gec_{C_{1}} 1.
\end{multline}
which is impossible for $\ve$ small enough depending on $C_{1}$, and so $V\cap B(x,r)\neq\emptyset$ for $\ve>0$ small.

But by \eqref{e:sigma-no-flat}, for all $r>0$ so that $B(x,r)\subseteq \bB$ (since $x\in 5A^{-1}\bB$), there is either $y\in B(x,r)\cap \Sigma$ so that $\dist(y,V)\geq \delta r$, or there is $y\in V\cap B(x,r)$ so that $\dist(y,\Sigma)\geq \delta r$. In the former case, if we let $0\leq \phi\leq 1 $ be a $\frac{1}{2r}$-Lipschitz function equal to $1$ on $B(x,2r)$ and zero outside $B(x,4r)$, then $\psi(z)=\phi(z)\dist(z,V)$ is a $6$-Lipschitz function (recall $\dist(z,V)\leq 8r$ for all $z\in B(x,4r)$ since $V\cap B(x,4r)\neq\emptyset$) and $\psi(z)\geq  \delta r/2$ on $B(y,\delta r/2)$. Hence,
\[
C_{1}^{-1} \delta^{d+1}2^{-d}
\leq r^{-d} \delta \sigma(B(y,r \delta/2))
\leq 2 r^{-d-1} \int \psi d\sigma 
\stackrel{\eqref{e:alpha4rV}}{<}\lec \ve,
\]
which is a contradiction for $\ve$ small enough. The case that there is $y\in V\cap B(x,r)$ so that $\dist(y,\Sigma)\geq \delta r$ has a similar proof and we omit it.

\end{proof}

\begin{lemma}
\label{l:dip}
Given $d\in \bN$, $C_{1}>0$, and $\beta>0$, there is $s_0<d$ so that the following holds. Let $s_{0}<s\leq d+1$ and $\Omega\subseteq \R^{d+1}$ be a connected domain so that \eqref{e:adr} and \eqref{e:bigbeta} hold. Let $B_0$ be a ball centered on $\d\Omega$ and $p\in \Omega\backslash aB_0$ (where $a=2b^{-1}>b^{-1}>1$ and $b$ is as in Lemma \ref{l:bourgain}). Set $\omega=\omega_{\Omega}^{p}$. Then for all $M_1>0$ there is $\delta>0$ and a ball $B\subseteq \frac{1}{2} B_0$ so that $r_{B}\geq \delta r_{B_0}$ and 
\[
\Theta_{\omega}^{s}(B)\not\in [M_1^{-1}\Theta_{\omega}^{s}(a B_0),M_1\Theta_{\omega}^{s}(a B_0)].\]
\end{lemma}

\def\tomega{\tilde{\omega}}
\begin{proof}
Without loss of generality, $B_0=\bB$. We can also assume that 
\begin{equation}
\label{e:dubassum}
\omega^{p}\ps{\frac{1}{10}\bB}\geq M_1^{-1}a^{-s} \omega^{p}(a\bB),
\end{equation}
otherwise we'd choose $B=\frac{1}{10}\bB$.

Let $\ve>0$ and 
\[
E=\{x\in \d B\cap \Omega: \delta_{\Omega}(x)>\ve\}.\]
Let $\tomega=\omega_{\Omega\backslash \bB}$. By the Strong Markov property\footnote{This follows from the Brownian motion definition of harmonic measure, but for a direct proof, see the appendix in \cite{AAM16}.} and for $\ve>0$ small enough,
\begin{align}
\label{e:wp/10}
\omega^{p}\ps{\frac{1}{10}\bB}
& =\int_{\d\bB\cap \Omega} \omega^{x}\ps{\frac{1}{10}\bB}d\tomega^{p}(x)\\
& \stackrel{\eqref{e:holder}}{\leq} C\ve^{\alpha}\tomega^{p}(\d\bB\cap \Omega)+\int_{E} \omega^{x}\ps{\frac{1}{10}\bB}d\tomega^{p}(x).\notag 
\end{align}
By Lemma \ref{l:bourgain}, $\omega^{x}(a\bB)\gec 1$ on $\d\bB\cap \Omega$, so by the maximum principle,
\begin{equation}
\label{e:maximum}
\tomega^{p}(\d\bB\cap \Omega) 
\lec \omega^{p}(a\bB) 
\stackrel{\eqref{e:dubassum}}{\leq} M_1a^{s}\omega^{p}(\frac{1}{10}\bB).
\end{equation}
So for $\ve>0$ small enough depending on $M_1$ and $a$, by \eqref{e:wp/10} and \eqref{e:maximum},
\begin{equation}
\label{e:omega1/2B}
\omega^{p}\ps{\frac{1}{10} \bB}
\leq 2\int_{E} \omega^{x}\ps{\frac{1}{10}\bB}d\tomega^{p}(x)
 .
\end{equation}
Let $B_j$ be a covering of $E$ by boundedly many balls centered on $E$ of radius $\ve/4$ (whose total number depends only on $\ve$ and $d$), so $2B_{j}\subseteq \Omega$. We claim we there are $t>0$ (depending on $\ve$ and $M_1$) and $j$ so that if $B'=B_{j}$, then
\[
\tomega^{p}(B')\geq t \omega^{p}\ps{\frac{1}{10} \bB} \mbox{ and }\omega^{x_{B'}}\ps{\frac{1}{10}\bB}\geq t.
\]
If not, then for each $j$ either $\tomega^{p}(B_j)< t \omega^{p}(\frac{1}{10} \bB)$ (let $J_{1}$ denote the set of these $j$) or $\omega^{x_{B_j}}(\frac{1}{10}\bB)< t$ (let $J_{2}$ denote the set of these $j$). For $j\in J_{2}$, Harnack's principle implies $\omega^{x}(\frac{1}{10}\bB)\lec  t$ for all $x\in B_j$. These alternatives and the fact that harmonic measure is at most $1$ imply
\begin{align*}
\omega^{p}\ps{\frac{1}{10}\bB}
& \stackrel{\eqref{e:omega1/2B}}{\lec} \ps{\sum_{j\in J_{1}}+\sum_{j\in J_{2}}} \int_{B_j\cap d\bB} \omega^{x}\ps{\frac{1}{10}\bB}d\tomega^{p}(x)\\
& \lec_{\ve} \sum_{j\in J_{1}}1\cdot  \tomega^{p}(B_j\cap \d \bB) + \sum_{j\in J_{2}} t \cdot \tomega^{p}(B_{j}\cap \d \bB) \\
& 
\lec_{\ve} t\omega\ps{\frac{1}{10} \bB}+t \tomega^{p}(\d \bB\cap \Omega) 
\end{align*}
which contradicts \eqref{e:maximum} for $t$ small enough.

Thus, we have a ball $B'$ centered on $\d\bB$ with $2B'\subseteq \Omega$ and $r_{B}=\ve/2$, and such that 
\begin{equation}
\label{e:tomegabig}
\tomega^{p}(B')\gec_{\ve,M_1} \omega^{p}\ps{\frac{1}{10}\bB}\stackrel{\eqref{e:dubassum}}{\geq} M_1^{-1}a^{-s} \omega^{p}(a\bB).
\end{equation}
and $\omega^{x_{B'}}(\frac{1}{10}\bB)\geq t$. Let $M>0$. Now since $B'\subseteq \bB\subseteq a\bB$, $x_{B'}$ is a $\frac{\ve}{2a}$-corkscrew point for $a\bB$, and so Lemma \ref{l:bigD-s=d} (applied with $aB$ in place of $B$, $A=10a$, $a^{s}M$ in place of $M$,  and $A_0=t$) implies there is $\delta>0$ depending on $M,d,C_{1}$, and $\beta$ and $B\subseteq 5A^{-1}(a\bB) = \frac{1}{2}\bB$ with $r_{B}\geq \delta$ so that $\Theta_{\omega^{x_{B'}}}^{s}(B)> M$. 
By Harnack's inequality, $\Theta_{\omega^{x}}^{s}(B)\gec M$ for all $x\in B'$. Thus,
\begin{multline*}
\omega^{p}(B)
=\int_{\d\bB\cap \Omega} \omega^{x}(B)d\tomega^{p}(x)
\geq \int_{B'\cap \d\bB} \omega^{x}(B)d\tomega^{p}(x)\\
\gec\tomega^{p}(B')Mr_{B}^{s}
\stackrel{\eqref{e:tomegabig}}{\gec}_{\ve,M_1,a} \omega^{p}(a\bB)Mr_{B}^{s}. 
\end{multline*}
Thus, for $M$ large enough (depending on $\ve$ and $M_1$),  we have 
\[
\Theta_{\omega^{p}}^{s}(B)=\omega^{p}(B)r_{B}^{-s}>M_1a^{-s}\omega^{p}(a\bB) =M_1 \Theta_{\omega^{p}}^{s}(a\bB),
\] which proves the lemma.

\end{proof}

\section{Proof of the \maintheorem}

Let $C_{1},\beta,s,d,\beta,$ and $\Omega\subseteq \R^{d+1}$ be as in the \maintheorem where $s_0$ is the constant from Lemma \ref{l:dip}, and set $\omega=\omega_{\Omega}^{p}$ for some $p\in \Omega$. \\

We recall the following version of ``dyadic cubes" for metric spaces, first introduced by David \cite{Dav88} but generalized in \cite{Chr90} and \cite{HM12}.

 \begin{theorem}
Let $X$ be a doubling metric space. Let $X_{k}$ be a nested sequence of maximal $\rho^{k}$-nets for $X$ where $\rho<1/1000$ and let $c_{0}=1/500$. For each $n\in\bZ$ there is a collection $\cD_{k}$ of ``cubes,'' which are Borel subsets of $X$ such that the following hold.
\begin{enumerate}
\item For every integer $k$, $X=\bigcup_{Q\in \cD_{k}}Q$.
\item If $Q,Q'\in \cD=\bigcup \cD_{k}$ and $Q\cap Q'\neq\emptyset$, then $Q\subseteq Q'$ or $Q'\subseteq Q$.
\item For $Q\in \cD$, let $k(Q)$ be the unique integer so that $Q\in \cD_{k}$ and set $\ell(Q)=5\rho^{k(Q)}$. Then there is $\zeta_{Q}\in X_{k}$ so that
\begin{equation}\label{e:containment}
B_{X}(\zeta_{Q},c_{0}\ell(Q) )\subseteq Q\subseteq B_{X}(\zeta_{Q},\ell(Q))
\end{equation}
and $ X_{k}=\{\zeta_{Q}: Q\in \cD_{k}\}$.
\end{enumerate}
\label{t:Christ}
\end{theorem}

Let $\cD$ be the Christ-David cubes for $\d\Omega$. Fix $n_0$ so that $x_{0}\not\in aB_{Q}$ for all $Q\in \cD_{n_0}$. By rescaling, we can assume without loss of generality that $n_0=0$. 

Let 
\[
\Theta_{\omega}^{s}(Q)=\frac{\omega(Q)}{\sigma(Q)}.
\]
\begin{lemma}
\label{l:dyadic-dip}
Let $n\geq 0$, $M_{2}>0$, and $Q\in \cD_{n}$. There is $N_{Q}\in \bN$ so that $N_{Q}\lec_{M_{2},C_{1},\beta,d} 1$ and such that there is $Q'\in \cD_{n+N_{Q}}$ contained in $\frac{c_{0}}{2}B_Q$ so that
\[\Theta_{\omega}^{s}(Q')\not\in [M_2^{-1}\Theta_{\omega}^{s}(Q),M_2\Theta_{\omega}^{s}(Q)].
\]
\end{lemma}

\begin{proof}
Fix $N\in \bN$ large enough so that if $\tilde{Q}\in \cD_{n+N}$ is the cube with same center as $Q$, then
\begin{equation}
\label{e:aBinB_Q}
aB_{\tilde{Q}}\subseteq \frac{c_{0}}{2}B_{Q}
\end{equation}
where $a$ is as in Lemma \ref{l:dip}. Since $N$ is fixed and only depends on some fixed universal constants, we will not indicate when constants depend on it below. Clearly we can assume 
\begin{equation}
\label{e:dyadic-dubassum}
\Theta_{\omega}^{s}(\tilde{Q})\geq M_{2}^{-1}\Theta_{\omega}(Q),
\end{equation}
otherwise we'd pick $Q'=\tilde{Q}$. Thus,
\begin{equation}
\label{e:density-chain}
\Theta_{\omega}^{s}(B_{\tilde{Q}})
\stackrel{\eqref{e:containment}}{\gec}_{C_{1}} \Theta_{\omega}^{s}(\tilde{Q})
\stackrel{\eqref{e:dyadic-dubassum}}{\geq }M_{2}^{-1}\Theta_{\omega}^{s}(Q) \stackrel{\eqref{e:aBinB_Q}}{\gec}_{N,C_{1}} M_{2}^{-1}\Theta_{\omega}^{s}(aB_{\tilde{Q}}),
\end{equation}
so by Lemma \ref{l:dip}, for $M_1>0$ there is $\delta$ depending on $M_1,M_2,d,C_{1},\beta$ and $N$,  and there is $B\subseteq \frac{1}{2}B_{\tilde{Q}}$ centered on $\d\Omega$ (and thus also centered on $\tilde{Q}$) with $r_{B}\geq \delta \ell(\tilde{Q})$ for which 
\[
\Theta_{\omega}^{s}(B)\not\in [ M_{1}^{-1} \Theta_{\omega}^{s}(aB_{\tilde{Q}}),M_{1} \Theta_{\omega}^{s}(aB_{\tilde{Q}})].
\]
Suppose first that $\Theta_{\omega}^{s}(B)< M_{1}^{-1} \Theta_{\omega}^{s}(aB_{\tilde{Q}})$. Then we take $Q'$ to be the largest cube containing the center of $B$ that is also contained in $B$, so 
\begin{equation}
\label{e:Q'<deltaQ}
\ell(Q')\sim r_{B}\gec \delta \ell(\tilde{Q})
\end{equation}
Furthermore,
\[
\Theta_{\omega}^{s}(Q')
\stackrel{\eqref{e:containment}\atop \eqref{e:Q'<deltaQ}}{\lec}_{C_{1}} \Theta_{\omega}^{s}(B)
<M_{1}^{-1}\Theta_{\omega}^{s}(aB_{\tilde{Q}})\stackrel{\eqref{e:density-chain}}{\lec}_{C_{1}} \frac{M_{2}}{M_{1}} \Theta_{\omega}^{s}(\tilde{Q})
\]
and the lemma follows in this case by picking $M_{1}\gg M_{2}^{2}$. 

Now suppose that $\Theta_{\omega}^{s}(B) \geq M_{1} \Theta_{\omega}^{s}(aB_{\tilde{Q}})$. Let $N'$ be the largest integer for which $5\rho^{N'}< \ell(\tilde{Q})/4$. Then there are at most boundedly many cubes in $\cD_{N'}$ which cover $B\cap \d\Omega$, and one of them, call it $Q'$, must have $\omega(Q')\gec \omega(B)$. Since $Q'\cap B\neq\emptyset$ and $\ell(Q')=5\rho^{N'}<\ell(\tilde{Q})/4$, and $B\subseteq \frac{1}{2} B_{\tilde{Q}}$, we have  
 \[
 Q'\subseteq B_{\tilde{Q}} \stackrel{\eqref{e:containment}}{\subseteq} \frac{c_{0}}{2} B_{Q}.
 \]
 Also, we have $\ell(Q')\sim_{\delta} \ell(\tilde{Q})\sim_{N}\ell(Q)$ and
 \[
 \Theta_{\omega}^{s}(Q')
 \gec \Theta_{\omega}^{s}(B)
 \geq M_{1} \Theta_{\omega}^{s}(aB_{\tilde{Q}})
 \gec_{C_{1}} M_{1}\Theta_{\omega}^{s}(\tilde{Q})\stackrel{\eqref{e:density-chain}}{\geq} \frac{M_{1}}{M_{2}}\Theta_{\omega}^{s}(Q).
 \]
 Again, the lemma follows by picking $M_{1}\gg M_{2}^{2}$. In either case, since $\delta$ always depends on $d,\beta,C_{1}, M_1$ and $M_2$ (and because $M_1$ depends on $d,\beta,C_1,$ and $M_2$), we have $\ell(Q')\gec_{d,\beta,C_{1},M_{2}} \ell(Q)$, so if $N_{Q}$ is such that $Q'\in \cD_{n+N_{Q}}$, then $N_{Q}\lec_{d,\beta,C_{1},M_{2}} 1$, and we're done.
\end{proof}

We now let $\ve>0$ be small and let $N_{Q}$ denote the integer from the previous lemma applied when $M_{2}=\ve^{-1}$. Fix a cube $Q_{0}\in \cD_{0}$ and define families of sub-cubes of $Q_{0}$ in $\cD_{n}'$ inductively as follows. First let $\cD_{0}'=\{Q_{0}\}$, then if $\cD_{n}'$ has been defined and $R\in \cD_{n}'$, and $n_R$ is so that $R\in \cD_{n_R}$, let
\def\Ch{{\rm Ch}}
\[
\Ch(R)=\{Q\in \cD_{n_R+N_{R}}: Q\subseteq R\}
\]
and 
\[\cD_{n+1}'=\bigcup_{R\in \cD_{n}'} \Ch(R), \;\;\; \cD'=\bigcup \cD_{n}'.\]

\begin{lemma}
There is $\lambda=\lambda(d,C_{1},\beta)\in (0,1)$ so that for all $R\in \cD'$,
\begin{equation}
\label{e:bourgain-drop}
\sum_{Q\in \Ch(R)}\omega(Q)^{\frac{1}{2}} \sigma(Q)^{\frac{1}{2}}<\lambda \omega(R)^{\frac{1}{2}}\sigma(R)^{\frac{1}{2}}. 
\end{equation}
\end{lemma}

\begin{proof}
Let $R'$ be the cube obtained in Lemma \ref{l:dyadic-dip} applied to $Q=R$ with $M_{2}=16$. Suppose first that 
\begin{equation}
\label{e:drop}
\Theta_{\omega}^{s}(R')<\frac{1}{16}\Theta_{\omega}^{s}(R).
\end{equation}
By the Cauchy-Schwartz inequality,
\begin{align*}
\sum_{Q\in \Ch(R)\atop Q\neq R'}\omega(Q)^{\frac{1}{2}}  \sigma(Q)^{\frac{1}{2}}
& \leq \ps{\sum_{Q\in \Ch(R)} \omega(Q)}^{\frac{1}{2}}\ps{\sum_{Q\in \Ch(R)\atop Q\neq R'}\sigma(Q)}^{\frac{1}{2}}\\
& = \omega(R)^{\frac{1}{2}}(\sigma(R)-\sigma(R'))^{\frac{1}{2}}.
\end{align*}
Also, by \eqref{e:drop},
\begin{align*}
\omega(R')^{\frac{1}{2}}\sigma(R')^{\frac{1}{2}}
& = \ps{\frac{\omega(R')}{\sigma(R')}}^{\frac{1}{2}} \sigma(R') 
<\ps{\frac{1}{16}}^{\frac{1}{2}} \ps{\frac{\omega(R)}{\sigma(R)}}^{\frac{1}{2}} \sigma(R') \\
& =\frac{1}{4}\omega(R)^{\frac{1}{2}}\sigma(R)^{\frac{1}{2}} \frac{\sigma(R') }{\sigma(R)}.
\end{align*}
These two estimates and the fact that $(1-x)^{\frac{1}{2}}\leq 1-\frac{x}{2}$ for all $x\leq 1$ imply that for $\ve$ small,
\begin{align*}
\sum_{Q\in \Ch(R)}\omega(Q)^{\frac{1}{2}} \sigma(Q)^{\frac{1}{2}}
& \leq \omega(R)^{\frac{1}{2}}(\sigma(R)-\sigma(R'))^{\frac{1}{2}}
+\frac{1}{4} \omega(R)^{\frac{1}{2}}\sigma(R)^{\frac{1}{2}} \frac{\sigma(R') }{\sigma(R)}\\
& =  \omega(R)^{\frac{1}{2}}\sigma(R)^{\frac{1}{2}}
\ps{ \ps{1-\frac{\sigma(R')}{\sigma(R)}}^{\frac{1}{2}} + \frac{1}{4}\frac{\sigma(R')}{\sigma(R)}}\\
& \leq \omega(R)^{\frac{1}{2}} \sigma(R)^{\frac{1}{2}}\ps{1-\frac{1}{4}\frac{\sigma(R')}{\sigma(R)}}
\end{align*}
Now \eqref{e:bourgain-drop} follows since, for $R\in \cD_{n}'$,  since $N_{R'}\lec_{\beta,C_{1},d}1$, we have
\begin{equation}
\label{e:sigmaratio}
\frac{\sigma(R')}{\sigma(R)}
\gec_{C_{1}} \frac{\ell(R')^{s} }{\ell(R)^{s}}
= \ps{ \frac{\ell(R')}{\ell(R)}}^{s}
\geq \ps{ \frac{\ell(R')}{\ell(R)}}^{d}
\gec_{C_{1},\beta,d} 1 .
\end{equation}
Hence, there is $t=t(C_{1},\beta,d)>0$ so that 
\[
1-\frac{1}{4}\frac{\sigma(R')}{\sigma(R)}
\leq 1-\frac{t}{4}=:\lambda<1
\]
and the lemma follows in this case.\\

Now suppose that
\begin{equation}
\label{e:bloop}
\Theta_{\omega}^{s}(R')>16 \Theta_{\omega}^{s}(R).
\end{equation}
By the Cauchy-Schwartz inequality,
\begin{align*}
\sum_{Q\in \Ch(R)\atop Q\neq R'}\omega(Q)^{\frac{1}{2}}  \sigma(Q)^{\frac{1}{2}}
& \leq \ps{\sum_{Q\in \Ch(R)\atop Q\neq R'} \omega(Q)}^{\frac{1}{2}}\ps{\sum_{Q\in \Ch(R)}\sigma(Q)}^{\frac{1}{2}}\\
& = (\omega(R)-\omega(R'))^{\frac{1}{2}}\sigma(R)^{\frac{1}{2}}.
\end{align*}

By \eqref{e:bloop},
\begin{align*}
\omega(R')^{\frac{1}{2}}\sigma(R')^{\frac{1}{2}}
& = \omega(R') \ps{\frac{\sigma(R')}{\omega(R')}}^{\frac{1}{2}} 
<\ps{\frac{1}{16}}^{\frac{1}{2}} \omega(R') \ps{\frac{\sigma(R)}{\omega(R)}}^{\frac{1}{2}} \\
& =\frac{1}{4}\omega(R)^{\frac{1}{2}}\sigma(R)^{\frac{1}{2}} \frac{\omega(R') }{\omega(R)}.
\end{align*}
Just as earlier, we have 
\begin{align*}
\sum_{Q\in \Ch(R)}\omega(Q)^{\frac{1}{2}} \sigma(Q)^{\frac{1}{2}}
& \leq (\omega(R)-\omega(R'))^{\frac{1}{2}}\sigma(R)^{\frac{1}{2}}
+\frac{1}{4} \omega(R)^{\frac{1}{2}}\sigma(R)^{\frac{1}{2}} \frac{\omega(R') }{\omega(R)}\\
& \leq  \omega(R)^{\frac{1}{2}}\sigma(R)^{\frac{1}{2}}
\ps{ \ps{1-\frac{\omega(R')}{\omega(R)}}^{\frac{1}{2}} + \frac{1}{4}\frac{\omega(R')}{\omega(R)}}\\
& \leq \omega(R)^{\frac{1}{2}} \sigma(R)^{\frac{1}{2}}\ps{1-\frac{1}{4}\frac{\omega(R')}{\omega(R)}}
\end{align*}

Now we use the fact that 
\[
\frac{\omega(R')}{\omega(R)}
\stackrel{\eqref{e:sigmaratio}}{\sim}_{C_{1},\beta,d} \frac{\Theta_{\omega}^{s}(R')}{\Theta_{\omega}^{s}(R)}
\stackrel{\eqref{e:bloop}}{\geq}  1.
\]
Thus, there is $t=t(C_{1},\beta,d)>0$ so that $\frac{\omega(R')}{\omega(R)}>t$. Hence, we again have 
\[
1-\frac{1}{4}\frac{\omega(R')}{\omega(R)}
\leq 1-\frac{t}{4}=:\lambda<1
\]
and again the lemma follows.

\end{proof}
Let $R\in \cD'$. Then for $Q\in \Ch(R)$ and for some $c=c(\beta,d,C_{1})$,
\[
\sigma(Q)\geq c \sigma(R).\]
Let $\tau\in (0,1) $ be small, we will fix its value later. Then
\begin{align*}
\sum_{Q\in \Ch(R)}\omega(Q)^{\frac{1}{2}} \sigma(Q)^{\frac{1-\tau}{2}}
& \leq \sigma(R)^{-\frac{\tau}{2}}c^{-\frac{\tau}{2}}\sum_{Q\in \Ch(R)}\omega(Q)^{\frac{1}{2}} \sigma(Q)^{\frac{1}{2}}\\
& \stackrel{\eqref{e:bourgain-drop}}{<} c^{-\frac{\tau}{2}}\lambda \omega(R)^{\frac{1}{2}}\sigma(R)^{\frac{1-\tau}{2}}
:=\gamma\omega(R)^{\frac{1}{2}}\sigma(R)^{\frac{1-\tau}{2}}
\end{align*}
Pick $\tau>0$ small enough so we still have that $\gamma:=c^{-\frac{\tau}{2}}\lambda<1$. 

Let 
\[
\cE_{n}=\{Q\in \cD_{n}': \omega(Q)\leq \sigma(Q)^{1-\tau}\}.
\]
Then
\begin{multline*}
\sum_{Q\in \cE_{n}}\omega(Q)
\leq \sum_{Q\in \cD_{n}'}\omega(Q)^{\frac{1}{2}}\sigma(Q)^{\frac{1-\tau}{2}}
 \leq \gamma\sum_{R\in \cD_{n-1}'}\omega(R)^{\frac{1}{2}}\sigma(R)^{\frac{1-\tau}{2}}<\cdots \\
\cdots  < \gamma^{n} \omega(Q_{0})^{\frac{1}{2} }\sigma(Q_{0})^{\frac{1-\tau}{2}}.
\end{multline*}
In particular, if
\[
E:=\ck{x\in Q_{0}:\lim_{r\rightarrow 0} \sup_{x\in Q\atop \ell(Q)<r}\frac{\omega(Q)}{\sigma(Q)^{1-\tau}}\leq 1}
\subseteq \bigcap_{n\geq 0} \bigcup_{Q\in \cE_{n}}Q,
\]
then
\[
\omega(E)\leq \lim_{k\rightarrow\infty}\sum_{n=k}^{\infty} \gamma^{n} \omega(Q_{0})^{\frac{1}{2}} \sigma(Q_{0})^{\frac{1-\tau}{2}}=0.
\]
Thus, if 
\[
F_{Q_{0}}:=\lim_{r\rightarrow 0} \sup_{x\in Q\atop \ell(Q)<r}\frac{\omega(Q)}{\sigma(Q)^{1-\tau}}\geq 1,
\]
then
\[
\omega(Q_{0}\backslash F_{Q_{0}})=0.
\]
Let $\ell>0$. Since the Christ-David cubes partition $Q_0$, for each $x\in F_{Q_{0}}$ we may find $Q_x\ni x$ contained in $Q_0$ with $\ell(Q_x)<\ell$ so that $\frac{\omega(Q_x)}{\sigma(Q_x)^{1-\tau}}> 1/2$. Let $Q_{j}$ be the collection of maximal cubes from $\{Q_{x}:x\in F_{Q_{0}}\}$. Then because the $Q_j$ are disjoint and $\diam Q_{j}\leq \diam B_{Q_{j}}=2\ell(Q_{j})<2\ell$,
\begin{align*}
\cH^{s(1-\tau)}_{2\ell}(F_{Q_{0}})
& \leq \sum_{j} (\diam Q_{j})^{s(1-\tau)}
\sim_{C_{1},\tau} \sum_{j} \sigma(Q_j)^{1-\tau}\\
& <2 \sum_{j} \omega(Q_j)\leq 2\omega(Q_0).
\end{align*}
Letting $\ell\rightarrow 0$ gives $\cH^{s(1-\tau)}(F_{Q_{0}})<\infty$.

Since our choice of $Q_{0}\in \cD_{0}$ was arbitrary and $\cD_{0}$ partitions $\d\Omega$, this implies $\dim \omega\leq s(1-\tau)$. This finishes the proof of the \maintheorem\!.

\bibliographystyle{alpha}

\begin{thebibliography}{HMM{\etalchar{+}}17}


\bibitem[Aik08]{Aik08}
H.~Aikawa.
\newblock Equivalence between the boundary {H}arnack principle and the
  {C}arleson estimate.
\newblock {\em Math. Scand.}, 103(1):61--76, 2008.


\bibitem[AH08]{AH08}
H.~Aikawa and K.~Hirata.
\newblock Doubling conditions for harmonic measure in {J}ohn domains.
\newblock {\em Ann. Inst. Fourier (Grenoble)}, 58(2):429--445, 2008.



\bibitem[AAM16]{AAM16}
M.~Akman, J.~Azzam, and M.~Mourgoglou.
\newblock Absolute continuity of harmonic measure for domains with lower
  regular boundaries.
\newblock {\em arXiv preprint arXiv:1605.07291}, 2016.

\bibitem[AG01]{AG}
D.~H. Armitage and S.~J. Gardiner.
\newblock {\em Classical potential theory}.
\newblock Springer Monographs in Mathematics. Springer-Verlag London, Ltd.,
  London, 2001.
  
  

\bibitem[AHM{\etalchar{+}}16]{AHMMMTV16}
J.~Azzam, S.~Hofmann, J.~M. Martell, S.~Mayboroda, M.~Mourgoglou, X.~Tolsa, and
  A.~Volberg.
\newblock Rectifiability of harmonic measure.
\newblock {\em Geom. Funct. Anal.}, 26(3):703--728, 2016.


\bibitem[AM18]{AM18}
J.~Azzam and M.~Mourgoglou.
\newblock Tangent measures and absolute continuity of harmonic measure.
\newblock {\em Rev. Mat. Iberoam.}, 34(1):305--330, 2018.


\bibitem[AM15]{AM15}
J.~Azzam and M.~Mourgoglou.
\newblock Tangent measures and densities of harmonic measure.
\newblock {\em to appear in Rev. Math.}, 2015.

\bibitem[AMT17]{AMT17}
J.~Azzam, M.~Mourgoglou, and X.~Tolsa.
\newblock A two-phase free boundary problem for harmonic measure and uniform
  rectifiability.
\newblock {\em arXiv preprint arXiv:1710.10111}, 2017.

\bibitem[Anc86]{Anc86}
A.~Ancona.
\newblock On strong barriers and an inequality of {H}ardy for domains in {${\bf
  R}^n$}.
\newblock {\em J. London Math. Soc. (2)}, 34(2):274--290, 1986.

\bibitem[Bat96]{Bat96}
A.~Batakis.
\newblock Harmonic measure of some {C}antor type sets.
\newblock {\em Ann. Acad. Sci. Fenn. Math.}, 21(2):255--270, 1996.

\bibitem[Bat00]{Bat00}
A.~Batakis.
\newblock A continuity property of the dimension of the harmonic measure of
  {C}antor sets under perturbations.
\newblock {\em Ann. Inst. H. Poincar\'{e} Probab. Statist.}, 36(1):87--107,
  2000.

\bibitem[Bat06]{Bat06}
A.~Batakis.
\newblock Dimension of the harmonic measure of non-homogeneous {C}antor sets.
\newblock {\em Ann. Inst. Fourier (Grenoble)}, 56(6):1617--1631, 2006.

\bibitem[BE17]{BE17}
S.~Bortz and M.~Engelstein.
\newblock Reifenberg flatness and oscillation of the unit normal vector.
\newblock {\em arXiv preprint arXiv:1708.05331}, 2017.

\bibitem[Bou87]{Bou87}
J.~Bourgain.
\newblock On the {H}ausdorff dimension of harmonic measure in higher dimension.
\newblock {\em Invent. Math.}, 87(3):477--483, 1987.

\bibitem[Car85]{Car85}
L.~Carleson.
\newblock On the support of harmonic measure for sets of {C}antor type.
\newblock {\em Ann. Acad. Sci. Fenn. Ser. A I Math.}, 10:113--123, 1985.

\bibitem[Chr90]{Chr90}
M.~Christ.
\newblock A {$T(b)$} theorem with remarks on analytic capacity and the {C}auchy
  integral.
\newblock {\em Colloq. Math.}, 60/61(2):601--628, 1990.

\bibitem[Dav88]{Dav88}
G.~David.
\newblock Morceaux de graphes lipschitziens et int\'egrales singuli\`eres sur
  une surface.
\newblock {\em Rev. Mat. Iberoamericana}, 4(1):73--114, 1988.

\bibitem[GM08]{Harmonic-Measure}
J.~B. Garnett and D.~E. Marshall.
\newblock {\em Harmonic measure}, volume~2 of {\em New Mathematical
  Monographs}.
\newblock Cambridge University Press, Cambridge, 2008.
\newblock Reprint of the 2005 original.

\bibitem[HKM06]{HKM}
J.~Heinonen, T.~Kilpel{\"a}inen, and O.~Martio.
\newblock {\em Nonlinear potential theory of degenerate elliptic equations}.
\newblock Dover Publications, Inc., Mineola, NY, 2006.
\newblock Unabridged republication of the 1993 original.

\bibitem[HM12]{HM12}
T.~Hyt{\"o}nen and H.~Martikainen.
\newblock Non-homogeneous {$Tb$} theorem and random dyadic cubes on metric
  measure spaces.
\newblock {\em J. Geom. Anal.}, 22(4):1071--1107, 2012.

\bibitem[HMM{\etalchar{+}}17]{HMMTZ17}
S.~Hofmann, J-M. Martell, S.~Mayboroda, T.~Toro, and Z.~Zhao.
\newblock Uniform rectifiability and elliptic operators with small carleson
  norm.
\newblock {\em arXiv preprint arXiv:1710.06157}, 2017.

\bibitem[JW88]{JW88}
P.~W. Jones and T.H. Wolff.
\newblock Hausdorff dimension of harmonic measures in the plane.
\newblock {\em Acta Math.}, 161(1-2):131--144, 1988.

\bibitem[KT99]{KT99}
C.~E. Kenig and T.~Toro.
\newblock Free boundary regularity for harmonic measures and {P}oisson kernels.
\newblock {\em Ann. of Math. (2)}, 150(2):369--454, 1999.

\bibitem[LVV05]{LVV05}
J.~L. Lewis, G.~C. Verchota, and A.~L. Vogel.
\newblock Wolff snowflakes.
\newblock {\em Pacific J. Math.}, 218(1):139--166, 2005.

\bibitem[Mak85]{Mak85}
N.~G. Makarov.
\newblock On the distortion of boundary sets under conformal mappings.
\newblock {\em Proc. London Math. Soc. (3)}, 51(2):369--384, 1985.

\bibitem[Mat95]{Mattila}
P.~Mattila.
\newblock {\em Geometry of sets and measures in {E}uclidean spaces}, volume~44
  of {\em Cambridge Studies in Advanced Mathematics}.
\newblock Cambridge University Press, Cambridge, 1995.
\newblock Fractals and rectifiability.

\bibitem[Mey09]{Mey09}
D.~Meyer.
\newblock Dimension of elliptic harmonic measure of snowspheres.
\newblock {\em Illinois J. Math.}, 53(2):691--721, 2009.

\bibitem[MV86]{MV86}
N.~Makarov and A.~Volberg.
\newblock On the harmonic measure of discontinuous fractals.
\newblock 1986.

\bibitem[Pop98]{Pop98}
I.~Popovici.
\newblock {\em Rigidity and dimension of the harmonic measure on {J}ulia sets}.
\newblock ProQuest LLC, Ann Arbor, MI, 1998.
\newblock Thesis (Ph.D.)--Michigan State University.

\bibitem[Tol15]{Tol15}
X.~Tolsa.
\newblock Uniform measures and uniform rectifiability.
\newblock {\em J. Lond. Math. Soc. (2)}, 92(1):1--18, 2015.

\bibitem[UZ02]{UZ02}
M.~Urba\'{n}ski and A.~Zdunik.
\newblock Hausdorff dimension of harmonic measure for self-conformal sets.
\newblock {\em Adv. Math.}, 171(1):1--58, 2002.

\bibitem[Vol92]{Vol92}
AL~Volberg.
\newblock On the harmonic measure of self-similar sets on the plane.
\newblock In {\em Harmonic Analysis and Discrete Potential theory}, pages
  267--280. Springer, 1992.

\bibitem[Vol93]{Vol93}
A.~L. Volberg.
\newblock On the dimension of harmonic measure of {C}antor repellers.
\newblock {\em Michigan Math. J.}, 40(2):239--258, 1993.

\bibitem[Wol95]{Wol95}
T.~H. Wolff.
\newblock Counterexamples with harmonic gradients in {${\bf R}^3$}.
\newblock In {\em Essays on {F}ourier analysis in honor of {E}lias {M}. {S}tein
  ({P}rinceton, {NJ}, 1991)}, volume~42 of {\em Princeton Math. Ser.}, pages
  321--384. Princeton Univ. Press, Princeton, NJ, 1995.

\end{thebibliography}
\newcommand{\etalchar}[1]{$^{#1}$}
\def\cprime{$'$}

\end{document}